\documentclass[a4paper, 12pt]{amsart}

\usepackage[colorlinks,citecolor = red, linkcolor=blue,hyperindex]{hyperref}
\usepackage{euscript,eufrak,verbatim}
\usepackage[psamsfonts]{amssymb}
\usepackage[usenames]{color}
\usepackage[curve]{xypic}
\usepackage{graphicx}
\usepackage{mathrsfs}
\usepackage{calligra,amsmath,amssymb}
\usepackage{float}
\usepackage{bbm}
 \usepackage{color}

\newtheorem{thm}{Theorem}[section]
\newtheorem{lem}[thm]{Lemma}

\newtheorem{prop}[thm]{Proposition}
\newtheorem{cor}[thm]{Corollary}
\newtheorem{defn}[thm]{Definition}
\newtheorem{rem}[thm]{Remark}

\def\R{\mathbb{R}}
\def\N{\mathbb{N}}
\def\Z{\mathbb{Z}}

\def\E{\mathbb{E}}

\def\FF{F}
\def\O{\mathcal{O}}

\def\X{\mathcal{X}}

\def\C{\mathbb{C}}

\def\1{\mathbbm{1}}
\def\an {\text{\, and \,}}
\def\tr{\mathrm{tr}}

\def\GL{\mathrm{GL}}
\def\Mat{\mathrm{Mat}}
\def\Skew{\mathrm{Skew}}

\def\diag{\mathrm{diag}}

\def\supp{\mathrm{supp}}

\def\ORB{\mathscr{ORB}}

\def\Sym{\mathrm{Sym}}
\def\ch{\mathrm{char}}

\def\co{\mathrm{co}}

\def\XXint#1#2#3{{\setbox0=\hbox{$#1{#2#3}{\int}$ }
\vcenter{\hbox{$#2#3$ }}\kern-.6\wd0}}

\begin{document}

\title[Ergodic measures on infinite skew-symmetric  matrices]{Ergodic measures on  infinite skew-symmetric matrices over non-Archimedean local fields}


\author
{Yanqi Qiu}
\address
{Yanqi QIU: CNRS, Institut de Math{\'e}matiques de Toulouse, Universit{\'e} Paul Sabatier, 118 Route de Narbonne, F-31062 Toulouse Cedex 9, France}

\email{yqi.qiu@gmail.com}


\begin{abstract}
Let $F$ be a non-discrete non-Archimedean locally compact  field  such that the characteristic $\ch(F)\ne 2$ and let $\mathcal{O}_F$ be  the ring of integers in $F$. The main results of this paper are Theorem \ref{CT-skew} that classifies ergodic probability measures on the space $\mathrm{Skew}(\mathbb{N}, F)$ of infinite skew-symmetric matrices with respect to the natural action of the group $\mathrm{GL}(\infty,\mathcal{O}_F)$ and Theorem \ref{thm-cor}, that gives  an unexpected natural correspondence between the set of $\mathrm{GL}(\infty,\mathcal{O}_F)$-invariant Borel probability measures on $\mathrm{Sym}(\mathbb{N}, F)$ and the set of $\mathrm{GL}(\infty,\mathcal{O}_F) \times \mathrm{GL}(\infty,\mathcal{O}_F)$-invariant Borel probability measures on the space $\mathrm{Mat}(\mathbb{N}, F)$ of infinite  matrices over $F$.
\end{abstract}

\subjclass[2010]{Primary 37A35; Secondary 60B10, 60B15}
\keywords{ergodic measures, non-Archimedean locally compact fields, Vershik-Kerov ergodic method, Ismagilov-Olshanski multiplicativity,  orbital integrals, skew-symmetric matrices}

\maketitle


\setcounter{equation}{0}

\section{Main results}\label{sec-main}
The Vershik-Kerov ergodic method is immensely successful in studying the  ergodic probability measures with respect to a fixed action of an inductively compact group, see  \cite{Vershik-inf-group, VK81, VK-char, KV86, OV-ams96, BQ-sym} and references therein.  In the present note,  following the method in \cite{BQ-sym},  we continue the study of the ergodic measures on the space of infinite skew-symmetric matrices over a local field.  

Fix any non-discrete locally compact  non-Archimedean field $F$ such that the characteristic $\ch(F) \ne 2$, that is, in $F$, we have $2 \ne 0$. Let $| \cdot |$ denote the  {\em absolute value} on $F$, then  the ring  of integers in $F$  is given by $\O_F  = \{x\in F: |x| \le 1\}$.    The subset $ \mathfrak{m}= \{x\in \FF: | x| < 1\}$ is a maximal and  principle ideal of $\O_F$. Throughout the paper, we fix a generator  $\varpi$ of $\mathfrak{m}$, that is, $\mathfrak{m}= \varpi \O_F$.
 The quotient  $\O_F/\mathfrak{m}$ is  a finite field with  $q =p^f$  elements for some prime number $p$ and positive integer $f\in\N$.

For any $n\in\N$, denote by $\GL(n, \O_F)$ the general linear group over $\O_F$ of size $n\times n$. Let 
\begin{align*}
 \GL(\infty, \O_F) : = \lim\limits_{\longrightarrow} \GL(n, \O_F)
 \end{align*} 
 be the inductive limit group. Equivalently, $\GL(\infty, \O_F)$ is the group of  infinite invertible matrices $g$ with entries in $\O_F$ such that $g_{ij} = \delta_{ij}$ if  $i +j$ is large enough.

Let $\Skew(\N, F)$ denote the space of infinite skew-symmetric matrices over $F$, that is, 
 \begin{align*}
 \Skew(\N, F): & = \{ [X_{ij}]_{i, j\in\N} | X_{ij}\in F,  X_{ij}  = - X_{ji},  \forall  i, j \in \N\}. 
 \end{align*}
 Similarly, for any $n\in\N$, denote
\begin{align*}
\Skew(n, \FF): &= \{ X = [X_{ij}]_{1\le i, j\le n} | X_{ij}\in F, X  = - X^t \}.
 \end{align*}
Consider the group action of $\GL(\infty, \O_F)$ on $\Skew(\N, F)$ defined by
\begin{align*}
(g, A) \mapsto g A g^{t}, \quad g \in \GL(\infty, \O_F), A \in \Skew(\N, F). 
\end{align*}
Let $\mathcal{P}_{\mathrm{erg}}(\Skew(\N, F))$ denote the corresponding space of ergodic probability measures on $\Skew(\N, F)$, endowed with the induced weak topology.

Our main result is the classification of $\mathcal{P}_{\mathrm{erg}}(\Skew(\N, F))$.  Let us proceed to the precise statement.  Define 
\begin{align*}
\Delta: =\left\{ \mathbbm{k}= (k_j)_{j = 1}^\infty\Big|  k_j \in \Z \cup\{-\infty\}; k_1 \ge k_2 \ge \cdots \right\}. 
\end{align*}
The set $\Delta$ is equipped with the induced topology of the Tychonoff's product topology on  $(\Z\cup\{-\infty\})^\N$ by the inclusion $\Delta\subset (\Z\cup\{-\infty\})^\N$. To each sequence $\mathbbm{k} \in \Delta$, we now assign an ergodic  $\GL(\infty, \O_F)$-invariant probability measure on $\Skew(\N, F)$.

In what follows, we use the convention $\varpi^\infty=0$.  
\begin{defn}\label{defn-rm}
Given an element $\mathbbm{k} \in\Delta$, let $\mu_{\mathbbm{k}}$ be the probability distribution of the infinite skew-symmetric random matrix $A_{\mathbbm{k}}$ defined as follows.   Let 
 \begin{align*}
 X_i^{(n)},  \quad Y_i^{(n)},  \quad Z_{ij}, \quad i< j, n = 1, 2, \cdots
 \end{align*}
 be independent  random variables,  each sampled uniformly from $\O_F$.   
 
 Denote $k: = \lim\limits_{n\to\infty}k_n \in \Z\cup\{-\infty\}$ and set 
\begin{align*}
A_{\mathbbm{k}} : =\sum\limits_{n:\, k_n > k} \varpi^{-k_n} \Big [ X_i^{(n)} Y_j^{(n)} -  X_j^{(n)} Y_i^{(n)} \Big]_{i, j\in\N} +  \varpi^{-k}Z,
\end{align*}
where $Z: = [Z_{ij}]_{i,j\in\N}$ is defined by $Z_{ij} =- Z_{ji}$ for $i>j$ and $Z_{ii}=0$.
\end{defn}

\begin{thm}\label{CT-skew}
Assume that the characteristic $\ch(F)\ne 2$. Then the  map $\mathbbm{k} \mapsto \mu_{\mathbbm{k}}$ defines a homeomorphism between  $\Delta$ and $\mathcal{P}_{\mathrm{erg}}(\Skew(\N, F))$. 
\end{thm}

\begin{rem}
By Theorem \ref{CT-skew},  the space $\mathcal{P}_{\mathrm{erg}}(\Skew(\N, F))$  is weakly closed in the space of all Borel measures  on $\Skew(\N, F)$.
\end{rem}

The proof of Theorem \ref{CT-skew} is  similar to that of \cite[Theorem 1.2]{BQ-sym},  in this note  we only give an outline. 

For stating our second result, let us recall some notation from \cite{BQ-sym}.  Given any group action of a group $G$ on a Polish space $\X$, we denote by $\mathcal{P}_{\mathrm{inv}}^G(\X)$ or  $\mathcal{P}_{\mathrm{inv}}(\X)$  the corresponding set of invariant Borel probability measures on $\X$ and by $\mathcal{P}_{\mathrm{erg}}^G(\X)$ or  $\mathcal{P}_{\mathrm{erg}}(\X)$ the set of ergodic  Borel probability measures on $\X$.

Recall the definition
$$
 \Mat(\N, F):  = \{ X  = [X(i,j )]_{i,j\in\N}| X(i, j)\in F \text{ for all $i, j \in \N$}\}
$$
and the group action of $\GL(\infty, \O_F) \times \GL(\infty, \O_F)$  on $\Mat(\N, F)$:  
$$
((g_1, g_2), X) \mapsto g_1 X g_2^{-1}, \quad g_1, g_2 \in \GL(\infty, \O_F), X \in \Mat(\N, F).
$$

The map $\tau(X):= X - X^t$ from  $\Mat(\N, F)$ to $\Skew(\N, F)$ 
induces an affine map
\begin{align}\label{def-tau}
\tau_{*}:  \mathcal{P}_{\mathrm{inv}}(\Mat(\N, F)) \to \mathcal{P}_{\mathrm{inv}}(\Skew(\N, F))
\end{align}
by restricting  the pushforward map onto $\mathcal{P}_{\mathrm{inv}}(\Mat(\N, F))$. 

\begin{thm}\label{thm-cor}
Assume that the characteristic $\ch(F)\ne 2$. Then the map $\tau_{*}:  \mathcal{P}_{\mathrm{inv}}(\Mat(\N, F)) \to \mathcal{P}_{\mathrm{inv}}(\Skew(\N, F))$ is a homeomorphic affine isomorphism between topological convex sets. 
\end{thm}

Comparing Theorem \ref{CT-skew} with \cite[Theorem 1.2]{BQ-sym}, using the notation in \cite{BQ-sym},  we also  have 
\begin{cor}
There is a natural correspondence between the sets of spherical representations of the following two infinite dimensional Cartan motion groups: 
$$
\Mat(\infty, F) \rtimes (\GL(\infty, \O_F) \times \GL(\infty, \O_F)),
$$
$$
  \Skew(\infty, F) \rtimes \GL(\infty, \O_F).
$$
\end{cor}

\section{Preliminaries}\label{sec-pre}
\subsection{Notation}

For any Polish space $\X$, we denote by $\mathcal{P}(\X)$ the set of Borel probability measures on $\X$. If a sequence $(\mu_n)_{n\in\N}$ in $\mathcal{P}(\X)$ converges weakly to $\mu\in\mathcal{P}(\X)$, then we denote it by $\mu_n \Longrightarrow \mu$.

Let $e_{11}$ denote the matrix whose  $(1,1)$-coefficient is $1$ and all other coefficients are zeros. The size of the matrix $e_{11}$ will be clear from the context.  Let $J$ denote the following  skew-symmetric matrix
\begin{align*}
 J := \left[\begin{array}{cc} 0 & 1 \\ -1 & 0\end{array}\right]. 
\end{align*}
In what follows, if $n\ge 2$, then for any  matrix $A$ over $F$ of size $n\times n$,  we use the conventions:
$$
A J :  = A \cdot \diag(J, \underbrace{0, \cdots, 0}_{\text{$n-2$ times}}) \an  J A :  =\diag(J, \underbrace{0, \cdots, 0}_{\text{$n-2$ times}}) \cdot A.
$$

Recall  the definition of  the characteristic function for a Borel probability measure on $\Skew(\N, F)$.    First, fix any  character $\chi$ of $F$  such that 
\begin{align*}  
\text{  $\chi|_{\O_F} \equiv 1 $ and $\chi$ is  not constant on  $\varpi^{-1} \O_F$.}
\end{align*} 
Let $\Skew(\infty, \FF)$ denote the subspace of $\Skew(\N, F)$ consisting of matrices whose all but a finite number of  coefficients are zeros. The characteristic function $\widehat{\mu}$ for any probability measure $\mu \in \mathcal{P}(\Skew(\N, F))$ is defined on $\Skew(\infty,F)$ by
\begin{align*}
\widehat{\mu} (X): =  \int_{\Skew(\N,\FF)} \chi( \tr(XS)) \mu(dS).
\end{align*}

 The following  lemma is elementary. For reader's convenience, we include its routine proof in Appendix. 
\begin{lem}\label{lem-stan-form}
Every matrix $A \in \Skew(2 n, \FF)$ can be written in the form
\begin{align*}
A = g \cdot \diag (\varpi^{-k_1}J , \varpi^{-k_2}J , \cdots, \varpi^{-k_n}J ) \cdot g^t, \quad g \in \GL(2n, \O_F),
\end{align*}
where $k_1, \cdots, k_n\in\Z\cup\{-\infty\}$ and  $k_1 \ge k_2 \ge \cdots \ge k_n \ge - \infty$. \end{lem}


\section{Outline of the proof of Theorem \ref{CT-skew}}

\subsection{Invariance and Ergodicity}\label{sec-inv-erg}

\begin{thm}\label{thm-erg-1}
For any $\mathbbm{k}\in \Delta$,  the probability measure $\mu_{\mathbbm{k}} $  is ergodic $\GL(\infty, \O_F)$-invariant. Moreover, for any two elements $\mathbbm{k}, \mathbbm{k}' \in \Delta$, 
 \begin{align}\label{unique-statement}
\text{$\mu_{\mathbbm{k}} = \mu_{\mathbbm{k}'}$  if and only if $\mathbbm{k} = \mathbbm{k}'$.}
  \end{align}
\end{thm}

The fact that $\mu_{\mathbbm{k}}$'s are $\GL(\infty, \O_F)$-invariant is simple and its proof is similar to that of \cite[Proposition 3.1 and Proposition 3.9]{BQ-sym}. 

The assertion \eqref{unique-statement} is an immediate consequence of \cite[Lemma 5.2]{BQ-sym} and Proposition \ref{prop-explicit} below. 

For proving that $\mu_{\mathbbm{k}}$'s are $\GL(\infty, \O_F)$-ergodic, we use again an argument of Okounkov and Olshanski in \cite{OkOl}: that is, the ergodicity of $\mu_{\mathbbm{k}}$ is derived from the De Finetti Theorem. For the detail, the reader is referred to the proof of \cite[Theorem 3.5]{BQ-sym}. A slight difference appears here is:  instead of using \cite[Lemma 3.6]{BQ-sym}, we use the following Lemma \ref{lem-exchange}, whose proof is simple and routine.

\begin{lem}\label{lem-exchange}
The map $A \mapsto (A(2i-1, 2i))_{i=1}^\infty$ from $\Skew(\N, F)$ to $F^\N$ induces an affine embedding  
\begin{align*}
\Psi: \mathcal{P}_{\mathrm{inv}}(\Skew(\N, F)) \rightarrow \mathcal{P}^{S(\infty)}_{\mathrm{inv}}(F^\N),
\end{align*}
where $S(\infty):  = \bigcup_{n\in\N} S(n)$ is the union of the group $S(n)$ of permutations of $\{1, 2, \cdots, n\}$ and  $S(\infty)$ acts on $F^\N$ by permutations of coordinates.   
\end{lem}

\subsection{Computation of characteristic functions}

For any $\mathbbm{k}$ in $\Delta$, since $\mu_{\mathbbm{k}}$ is $\GL(\infty, \O_F)$-invariant, so does $\widehat{\mu_{\mathbbm{k}}}$.  Consequently,  it suffices to compute
 \begin{align*}
  \widehat{\mu_{\mathbbm{k}}} (  \diag (\varpi^{-\ell_1} J , \cdots, \varpi^{-\ell_r}J , 0, \cdots))
 \end{align*}
for any $r\in\N$ and any $\ell_1, \cdots, \ell_r \in\Z$.

Define a function $\Theta: F \rightarrow \C$ by 
\begin{align*}
\Theta(x) : =  \int\limits_{\Mat(2, \O_F)} \chi\Big(  x \cdot \tr (Y J Y^t J) \Big) dY.
\end{align*}
Since $\ch(F)\ne 2$, there exists a unique $\ell_0\in \N\cup\{0\}$, such that 
\begin{align}\label{2-ord}
|2| = q^{-\ell_0}. 
\end{align}

\begin{lem}\label{lem-theta}
The function $\Theta$ is given by 
\begin{align*}
\Theta(x)  = \exp(   -2  \log q \cdot (\ell-\ell_0) \1_{\ell  \ge \ell_0}),  \text{if $|x| =q^{\ell}$}.
\end{align*}
\end{lem}
\begin{proof}
By a direct computation, we have 
\begin{align*}
\Theta(x)  =  &\int\limits_{\Mat(2, \O_F)} \chi(  -2 x  \cdot \det Y ) dY
\\
=&  \int\limits_{\O_F^4} \chi(  -2 x (y_1y_2 - y_3y_4) ) dy_1dy_2 dy_3 dy_4
\\
 =& \int\limits_{\O_F^2} \chi(  -2 x y_1y_2)  dy_1dy_2\int\limits_{\O_F^2} \chi(  2 x  y_3y_4 ) dy_3 dy_4 
 \\
 =& \Big| \int\limits_{\O_F^2} \chi(  2 x y_1y_2  ) dy_1dy_2\Big|^2 
 = \left\{\begin{array}{cl} 1 &\text{if $| 2x|\le 1$} \\ \frac{1}{|2x|^2} & \text{if $| 2x| > 1$ }    \end{array} \right..
\end{align*}
Taking \eqref{2-ord} into account,  we get the desired identity.
\end{proof}

Using Lemma \ref{lem-theta} and by direct computations, we obtain the following
\begin{prop}\label{prop-explicit}
Let $\mathbbm{k} = (k_n)_{n\in\N}\in \Delta$. Then any $\ell \in \Z$, we have  
 \begin{align*}
 \widehat{\mu_{\mathbbm{k}}}(\varpi^{-\ell} e_{11}) =    \exp\Big(- 2 \log q \cdot   \sum_{n=1}^\infty (k_n+\ell-\ell_0)  \1_{\{  k_n+\ell  \ge \ell_0\}}\Big).
 \end{align*}
More generally, for  any $\ell_1, \cdots, \ell_r\in\Z$, we have 
 \begin{align*}
 \widehat{\mu_{\mathbbm{k}}}(\diag (\varpi^{-\ell_1}, \cdots, \varpi^{-\ell_r},  0, \cdots)) = \prod_{ i=1}^r    \widehat{\mu_{\mathbbm{k}}}(\varpi^{-\ell_i} e_{11}).
 \end{align*}
\end{prop}

\subsection{Vershik-Kerov ergodic method}
For notational convenience, we will use the identity
$$
\GL(\infty, \O_F) = \lim_{\longrightarrow} \GL(n, \O_F) = \lim_{\longrightarrow} \GL(2n, \O_F).
$$

\begin{defn}
 (i) For  any $n\in\N$ and any $A\in\Skew(\N, F)$, we define the $\GL(2n, \O_F)$-orbital measure generated by $A$,  denoted by $m_n(A)$, as the image of the normalized Haar measure of $\GL(2n, \O_F)$ under the map  $g \mapsto g \cdot x$ from $\GL(2n, \O_F)$ to $\Skew(\N, F)$.

(ii) For any $n\in\N$, let $\ORB_{n}(\Skew(2n, F))$ denote the set of all orbital measures $m_n(A)$,  where $A$ ranges over the space $\Skew(2n, F)$.
 
 (ii) Let $\ORB_\infty(\Skew(\N, F))$ be the set of probability measures $\mu$ on $\Skew(\N, F)$ such that there exists a sequence of positive integers $n_1 < n_2 < \cdots$ and  a sequence $(\mu_{n_k})_{k\in\N}$ of orbital measures  with  $\mu_{n_k} \in \ORB_{n_k} (\Skew(2 n_k, F)) $, so that  $\mu_{n_k} \Longrightarrow \mu$. 
\end{defn}

 As a variant of Vershik's Theorem, we have
\begin{thm}[Vershik]\label{Vershik-thm}
The following inclusion holds:
\begin{align*}
\mathcal{P}_{\mathrm{erg}}(\Skew(\N, F)) \subset \ORB_{\infty}(\Skew(\N, F)). 
\end{align*} 
\end{thm}

\subsection{Asymptotic formulas for orbital integrals}\label{sec-ch}

By Lemma \ref{lem-stan-form}, any $\mu\in \ORB_{n}(\Skew(2n, F))$ can be written in the form: 
\begin{align*}
\mu = m_n(\diag (x_1 J , \cdots, x_n J )).
\end{align*}
By the $\GL(2n, \O_F)$-invariance of $\widehat{\mu}$, it suffices to compute  the following orbital integrals: 
\begin{align}\label{def-DA}
   \widehat{m_n(D)} (A) = \int_{\GL(2n, \O_F)}    \chi   (\tr ( gD g^t  A  )) d g,
\end{align}
 where $D$ and $A$ are two skew-symmetric matrices given by: 
\begin{align}\label{D-an-A}
D = \diag (x_1 J , \cdots, x_n J ),  A = \diag(a_1 J , \cdots, a_r J , 0, \cdots),
\end{align}
with $ r \le n$ and $x_1, \cdots, x_n, a_1, \cdots, a_r \in F$.

\begin{thm}\label{thm-asy-mul}
Let $n, r \in \N$ be such that $r \le n $.  Suppose that $D$ and $A$ are two skew-symmetric matrices given by \eqref{D-an-A}. Then
\begin{align}\label{nsym-re}
&  \Big|  \widehat{m_n(D)} (A)  -  \prod_{i = 1}^r \prod_{j=1}^n \Theta (a_i x_j)   \Big| \le  2  \cdot  (1  - \prod_{w = 0}^{2r-1} (1 - q^{w-2n})), 
 \end{align}
where $dg$ is the normalized Haar measure on  $\GL(2n, \O_F)$. 
\end{thm}

The proof of Theorem \ref{thm-asy-mul} is based on the following

\begin{prop}\label{prop-gl-mat}
Let $f: \Mat(n, \O_F)\rightarrow \C$ be a bounded measurable function. Assume that $f$ depends only on the  the $r \times n$ coefficients, that is, we have
\begin{align*}
f(X) = f([X_{ij}]_{1\le i \le r, 1\le j \le n}). 
\end{align*}
Then 
\begin{align*}
\Big| \int\limits_{\GL(n, \O_F)} f(g) dg -  \int\limits_{\Mat(n, \O_F)} f(X) dX\Big| \le 2  \|f\|_\infty \cdot  (1  - \prod_{w = 0}^{r-1} (1 - q^{w-n})).   
\end{align*}
where $dg$ and $dX$ are normalized Haar measures of $\GL(n, \O_F)$ and $\Mat(n, \O_F)$ respectively. 
\end{prop}
\begin{proof}
The proof of Proposition \ref{prop-gl-mat} is similar to that of \cite[Theorem 7.1]{BQ-sym}, here we omit the detail.
\end{proof}
\begin{rem}
Proposition \ref{prop-gl-mat} goes back to the author's note \cite{Q-trunc}.  
\end{rem}

\begin{proof}[Proof of Theorem \ref{thm-asy-mul}]
Fix $n, r \in\N$ and fix the two skew-symmetric matrices $D$ and $A$ given as in \eqref{D-an-A}. 
Define $f: \Mat(2n, \O_F)\rightarrow \C$ by the formula 
\begin{align*}
f(X)=  \chi   (\tr ( XD X^t  A  )). 
\end{align*}
Writing $X\in \Mat(2n, \O_F)$ in the form $X= [X_{ij}]_{1\le i, j \le n}$, 
where $X_{ij}$'s  are $2\times 2$ blocks. 
Expanding $\tr ( XD X^t  A  )$, we get
\begin{align*}
f(X)=  \prod_{i = 1}^r \prod_{j=1}^n \chi\Big(  a_i x_j  \cdot     \tr (X_{ij} J X_{ij}^t J)  \Big).
\end{align*}
This implies that $f$ only depends on the $2r \times 2n$ coefficients of $X$. Applying  Proposition \ref{prop-gl-mat}, we get 
\begin{align*}
 \Big| \int\limits_{\GL(2n, \O_F)}  f(g) dg -  \int\limits_{\Mat(2n, \O_F)} f(X) dX\Big| \le 2  \cdot  (1  - \prod_{w = 0}^{2r-1} (1 - q^{w-2n})).   
\end{align*}
The inequality \eqref{D-an-A} follows by observing that
\begin{align*}
&\int\limits_{\Mat(2n, \O_F)} f(X)  dX   =  \int\limits_{\Mat(2n, \O_F)}  \prod_{i = 1}^r \prod_{j=1}^n    \chi\Big(a_i x_j  \cdot     \tr (X_{ij} J X_{ij}^t J)  \Big) dX
\\
 &=  \prod_{1\le i \le r, 1 \le j \le n}   \int\limits_{\Mat(2, \O_F)}\chi\Big( a_i x_j  \cdot     \tr (YJ Y^t J)  \Big) dY
 =   \prod_{1\le i \le r, 1 \le j \le n} \Theta(a_i x_j). 
\end{align*}
\end{proof}


\subsection{Completion of classification of $\mathcal{P}_{\mathrm{erg}}(\Skew(\N, F))$}

\begin{lem}\label{lem-compact}
Let $(\mu_n)_{n\in\N}$ be a sequence of probability measures such that  $\mu_n\in\ORB_{n}(\Skew(2n, F))$. A necessary and sufficient condition for  $(\mu_n)_{n\in\N}$ to be  tight is the following:  
\begin{itemize}
\item[($C$)] There exists $\gamma \in\Z$ such that the supports $\supp(\mu_n)$ are all contained in the following compact subset of $\Skew(\N, F)$:   
\begin{align*}
\left\{X  \in \Skew(\N, F) \Big |   \text{$ | X_{ij}| \le q^\gamma, \forall i, j \in \N$}  \right\}.
\end{align*}
\end{itemize}
\end{lem}

\begin{proof}
The proof of Lemma \ref{lem-compact} is similar to that of \cite[Lemma 8.6]{BQ-sym}.
\end{proof}

\begin{thm}\label{ct-ns}
The  map $\mathbbm{k} \mapsto \mu_{\mathbbm{k}}$ induces a bijection between  $\Delta$ and $\mathcal{P}_{\mathrm{erg}}(\Skew(\N, F))$. 
\end{thm}

\begin{proof}
The proof is similar to that of \cite[Theorem 8.8]{BQ-sym}.
\end{proof}

\begin{proof}[Proof of Theorem \ref{CT-skew}]
By Theorem \ref{ct-ns}, we only need to prove that the map $\mathbbm{k} \mapsto \mu_{\mathbbm{k}}$ from $\Delta$ to $\mathcal{P}_{\mathrm{erg}} (\Skew(\N, F))$ and its inverse are both continuous.  This part of proof is similar to that of \cite[Theorem 1.2]{BQ-sym}.
\end{proof}

\section{Correspondence  between $\mathcal{P}_{\mathrm{inv}}(\Mat(\N, F))$ and $\mathcal{P}_{\mathrm{inv}}(\Skew(\N, F))$ }

In this section,  we give the proof of Theorem \ref{thm-cor}. First, we derive a Krein-Milman type result from the ergodic decomposition formula due to Alexander Bufetov for invariant probability measures with respect to a fixed  action of inductively compact group.  Let 
\begin{align*}
K(1) \subset K(2) \subset \cdots \subset K(n) \subset \cdots
\end{align*}
be an increasing chain of compact metrizable groups and set
\begin{align*}
K(\infty) =\bigcup_{n\in\N} K(n). 
\end{align*}
Fix a group action of $K(\infty)$  on a Polish space $\X$. 

\begin{thm}[{Bufetov \cite[Theorem 1]{Bufetov-erg-dec}}]\label{Buf-thm}
For any $\nu\in\mathcal{P}_{\mathrm{inv}}^{K(\infty)}(\X)$, there exists a  Borel probability $\overline{\nu}$ on $\mathcal{P}_{\mathrm{erg}}^{K(\infty)}(\X)$ such that 
\begin{align}\label{erg-dec}
\nu  = \int\limits_{\mathcal{P}_{\mathrm{erg}}^{K(\infty)}(\X)} \eta \,  d\overline{\nu} (\eta).  
\end{align}
Clearly, any Borel probability measure $\nu$ represented by the formula \eqref{erg-dec} is $K(\infty)$-invariant. 
\end{thm}

Given a subset $E\subset V$ of a locally convex topological vector space $V$, we denote by $\overline{\co}[E]$ the closed convex hull of $E$.  If moreover,  $K(\infty)$ acts on $\X$ by homeomorphisms, then it is easy to see that  $\mathcal{P}_{\mathrm{inv}}^{K(\infty)}(\X)$  is weakly closed. As an immediate corollary of Bufetov's Theorem \ref{Buf-thm}, we have

\begin{cor} 
Assume that  $K(\infty)$  acts on a Polish space $\X$ by homeomorphisms. Then 
$\mathcal{P}_{\mathrm{inv}}^{K(\infty)}(\X)  = \overline{\co}[\mathcal{P}_{\mathrm{erg}}^{K(\infty)}(\X)]$.
\end{cor}

Let us define $\mathcal{M}_{\mathrm{inv}}(\Skew(\N, F))$ to be the Banach space over $\R$ consisting  of $\GL(\infty, \O_F)$-invariant signed Borel measures {\it of finite total variation} on $\Skew(\N, F)$ and equipped with the norm: 
$$
\|\mu\| = | \mu| = \text{total variation of $\mu$}.
$$
 Define  the Banach space $\mathcal{M}_{\mathrm{inv}}(\Mat(\N, F))$ in a similar way. Note that by Jordan decomposition theorem for signed measures, any measure $\sigma\in \mathcal{M}_{\mathrm{inv}}(\Mat(\N, F))$ is represented by 
$$
\sigma = \lambda_1 \mu_1 - \lambda_2 \mu_2, \text{ with } \lambda_1, \lambda_2 \ge 0, \mu_1, \mu_2 \in \mathcal{P}_{\mathrm{inv}}(\Mat(\N, F)). 
$$

\begin{proof}[Proof of Theorem \ref{thm-cor}]
By the classification of $\mathcal{P}_{\mathrm{erg}}(\Skew(\N, F))$  in Theorem \ref{CT-skew}  and the classification $\mathcal{P}_{\mathrm{erg}}(\Mat(\N, F))$ in \cite{BQ-sym}, we see immediately that the following  affine map defined in \S \ref{sec-main}: 

\begin{align}\label{tau-pf}
\tau_{*}:  \mathcal{P}_{\mathrm{inv}}(\Mat(\N, F)) \to \mathcal{P}_{\mathrm{inv}}(\Skew(\N, F))
\end{align}
induces a homeomorphism
$$
\tau_{*}:  \mathcal{P}_{\mathrm{erg}}(\Mat(\N, F)) \to \mathcal{P}_{\mathrm{erg}}(\Skew(\N, F)). 
$$
Consequently, we have the following bijection
\begin{align*}
(\tau_{*})_{*}:  \mathcal{P} (\mathcal{P}_{\mathrm{erg}}(\Mat(\N, F)))  \to \mathcal{P}( \mathcal{P}_{\mathrm{erg}}(\Skew(\N, F))). 
\end{align*}
Applying Theorem \ref{Buf-thm}, we obtain that the map  \eqref{tau-pf} is a bijection. 

Now we prove that the map \eqref{tau-pf} is homeomorphic. The bijective map \eqref{tau-pf} extends to a bijective linear map (denoted again by $\tau_{*}$): 
\begin{align}\label{tau-extend}
\tau_{*}:  \mathcal{M}_{\mathrm{inv}}(\Mat(\N, F)) \to \mathcal{M}_{\mathrm{inv}}(\Skew(\N, F))
\end{align}
Moreover, since the map $X\stackrel{\tau}{\mapsto}X-X^t$ is continuous from  $\Mat(\N, F)$ to $\Skew(\N, F)$, so is the induced map  \eqref{tau-extend}. Applying the Open Mapping Theorem, we obtain that the map  \eqref{tau-extend} is a homeomorphism. It follows that the restriction \eqref{tau-pf} is also homeomorphic. The proof of Theorem \ref{thm-cor} is completed. 
\end{proof}

{\flushleft \bf Comments on Theorem \ref{thm-cor}.}
(i) The injectivity of the map \eqref{tau-pf} can be proved directly in an alternative way as follows. Assume that $\sigma \in \mathcal{P}_{\mathrm{inv}}(\Mat(\N, F))$. Let us show  that  $\widehat{\sigma}$ is uniquely determined by $\tau_{*}(\sigma)$. 

First let us compute 
$$
\widehat{\sigma}(\diag(x, y, 0, \cdots)), \quad x, y \in F^{\times}.
$$  
By the invariance of $\sigma$, the  function $\widehat{\sigma}(\diag(x, y, 0, \cdots))$ is symmetric on $x, y$. Without loss of generality, we may assume that $0 < |y| \le |x|$. Then $x^{-1} y\in \O_F\setminus \{0\}$. We claim that 
\begin{align}\label{sigma-tau-sigma}
\widehat{\sigma}(\diag(x, y, 0, \cdots))=\widehat{\tau_{*}(\sigma)} (2^{-1} x J).
\end{align}
Indeed, the following equality 
\begin{align*}
\left[\begin{array}{cc} x & 0\\ 0& y\end{array}\right]  = \left[\begin{array}{cc}0 & x\\ -x& 0\end{array}\right] \left[\begin{array}{cc}0 & -x^{-1}y\\1 &0 \end{array}\right]
\end{align*}
combined with the $\GL(\infty, \O_F)\times \GL(\infty, \O_F)$-invariance of $\sigma$ implies that  
\begin{align}\label{diag-yJ}
\widehat{\sigma}(\diag(x, y, 0, \cdots)) = \widehat{\sigma} \left( \left[\begin{array}{cc}0 & x\\ -x& 0\end{array}\right] \left[\begin{array}{cc}0 & -x^{-1}y\\1 &0 \end{array}\right]\right)  = \widehat{\sigma} (xJ ). 
\end{align}
Take  an infinite random matrix $M$ in $\Mat(\N, F)$ sampled with respect to $\sigma$, then the distribution of the random matrix $M - M^t$  is $\tau_{*}(\sigma)$.  We can write \eqref{diag-yJ} in the form 
\begin{align}\label{exp-form}
\begin{split}
&\E[\chi(M(1,1) x + M(2, 2) y)] = \E[\chi((M(1,2) - M(2,1)) \cdot x)]
\\
& = \E [\chi(  2^{-1} x \cdot \tr ((M - M^t) J)) ]. 
\end{split}
\end{align}
This is exactly the desired equality \eqref{sigma-tau-sigma}. 

Taking $y\to 0$ in \eqref{exp-form},  by bounded convergence theorem, we get 
\begin{align*}
\widehat{\sigma}(x e_{11})=\widehat{\tau_{*}(\sigma)} (2^{-1} x J).
\end{align*}
Similarly, for  general $r\in\N$, we can write the function 
\begin{align*}
\widehat{\sigma}(\diag(x_1, \cdots, x_r, 0, \cdots)), \quad x_1, \cdots, x_r \in F
\end{align*}
in terms of  the characteristic function of $\tau_{*}(\sigma)$.  

We thus prove the injectivity of the map \eqref{tau-pf}. 

(ii) For the natural group action $\GL(\infty, \O_F)$ on the space $\Sym(\N, F)$ of infinite symmetric matrices over $F$, there is not an analogue result as Theorem \ref{thm-cor} for $\mathcal{P}_{\mathrm{inv}}(\Mat(\N, F))$ and  $\mathcal{P}_{\mathrm{inv}}(\Sym(\N, F))$. 

 However, similar argument as in above  shows that if $\ch(F)\ne 2$, then the map $X\mapsto X + X^t$ from $\Mat(\N, F)$ to $\Sym(\N, F)$ induces an affine embedding 
\begin{align*}
 \mathcal{P}_{\mathrm{inv}}(\Mat(\N, F)) \hookrightarrow \mathcal{P}_{\mathrm{inv}}(\Sym(\N, F)). 
\end{align*} 

(iii) Note that in Euclidean case, for any size, there exists a trivial correspondence between the space of Hermitian matrices  and the space of anti-Hermitian matrices. Hence we have a trivial correspondence of the set of unitarily invariant Borel probability measures on the space of infinite Hermitian matrices and the set of Borel probability measures on the space of infinite anti-Hermitian matrices. 

\begin{rem}
It is not clear whether one can prove, without using the classifications of ergodic measures, that the map \eqref{tau-pf} is surjective. 
\end{rem}

\section{Appendix}

\begin{proof}[Proof of Lemma \ref{lem-stan-form}]
First note that for any $x  = u \varpi^{-k}\in F$ with $u\in \O_F^{\times}$ and $k \in \Z\cup\{-\infty\}$, we have
\begin{align*}
xJ =  \left[ \begin{array}{cc}0 & x \\ -x & 0 \end{array}\right] =  \left[ \begin{array}{cc}u & 0 \\  0& 1 \end{array}\right] \left[ \begin{array}{cc}0 & \varpi^{-k} \\ -\varpi^{-k} & 0 \end{array}\right] \left[ \begin{array}{cc}u & 0 \\ 0 & 1 \end{array}\right]. 
\end{align*}
It follows that we only need to prove that each skew-symmetric matrix $A\in\Skew(2n, F)$ can be written in the form 
\begin{align}\label{A-dec-bis}
A = g \cdot \diag (x_1J , x_2 J , \cdots, x_n J ) \cdot g^t, \quad (g \in \GL(2n, \O_F)).
\end{align}
If $n=1$, then there is nothing to prove. Assume that the decomposition \eqref{A-dec-bis} holds for $n-1$. Now take any non-zero skew-symmetric matrix  $A  = [x_{ij}]_{1\le i, j \le 2n}\in \Skew(2n, F)$. We claim that, upper to passing to $g \cdot A \cdot g^t$ for  a certain $g \in \GL(2n, \O_F)$ if necessary, we may assume that 
\begin{align*}
|x_{12}| = \max_{1\le i, j \le 2n}|x_{ij}|.
\end{align*}
Indeed, assume that $(i_0, j_0)$  is a pair of indices such that $i_0 < j_0$ and $|x_{i_0j_0}  | = \max_{1\le i, j \le 2n} | x_{ij}|.$ If $(i_0, j_0) = (1, 2)$, then there is nothing to do. 
Now assume that  $(i_0, j_0)\ne (1, 2)$. For a given permutation $\sigma \in S(2n)$, we denote by $M_{\sigma}$ the permutation matrix defined by $M_{\sigma} (i, j)= \1_{\sigma(i) =j}$.  If $\{i_0, j_0\} \cap\{1, 2\} = \emptyset$, then let $g =  M_{(1 i_0)} M_{(2j_0)}$, where $(1 i_0)$ and $(2j_0)$ are the transpositions exchanging $1, i_0$ and $2, j_0$ respectively. Then
\begin{align*}
g A g^t = M_{(1 i_0)} M_{(2j_0)} A M_{(2 j_0)}^t M_{(1 i_0)}^t = M_{(2j_0)} M_{(1 i_0)}   A M_{(2 j_0)} M_{(1 i_0)}
\end{align*}
and the $(1,2)$-coefficient of $g A g^t$ is $x_{i_0j_0}$.   If $\{i_0, j_0\} \cap\{1, 2\} \ne \emptyset$. Either $\{i_0, j_0\} \cap\{1, 2\} = \{1\}$,  since $j_0 > i_0$, we must have $i_0 = 1$ or $\{i_0, j_0\} \cap\{1, 2\} = \{2\}$, which implies that  $i_0 = 2$, indeed, otherwise, we would have $j_0=2$, then since $i_0 < j_0$, we would have $(i_0, j_0) = (1, 2)$ contradicts to the assumption that $(i_0, j_0)\ne (1,2)$.  In the first case $i_0=1$, take $g  = M_{(1j_0)}$, then the $(1, 2)$-coefficient of $gAg^t = M_{(2j_0)} A M_{(2j_0)}$ is $x_{i_0j_0}$.  In the second case $i_0=2$, take $g = M_{(2j_0)}M_{(12)}$, then 
\begin{align*}
g Ag^t  = M_{(2j_0)}M_{(12)} AM_{(12)}   M_{(2j_0)} 
\end{align*}
and the $(1,2)$-coefficient of $g A g^t$ is $x_{i_0j_0}$.  

Now we may write $A$ in the block form $A = [A_{uv}]_{1\le u, v\le n}$ with $A_{uv}$ the $2\times 2$ blocks. By the above assumption, we have $A_{11} = xJ$ and $x$ has maximal absolute value.  By taking 
\begin{align*}
g^t = \left[\begin{array}{ccccc} 1 & x^{-1} J A_{12} & x^{-1} J A_{13} & \cdots & x^{-1} JA_{1n}  \\ 0 & 1 & 0 & \cdots & 0 \\ 0& 0 & 1 & \cdots &0 \\ \vdots & \vdots & \vdots & \ddots& \vdots \\ 0 & 0 & 0 & \cdots& 1\end{array}\right] \in \GL(2n, \O_F), 
\end{align*}
we arrive at 
\begin{align*}
gA g^t =\left[\begin{array}{cc} xJ & 0 \\ 0 & A' \end{array}\right]. 
\end{align*}
By induction assumption, the $(2n-2)\times (2n-2)$ skew-symmetric matrix  $A'$ is diagonalizable. Hence so does $A$. By induction, Lemma \ref{lem-stan-form} is proved completely. 
\end{proof}


\section{Acknowledgements}
The author is deeply grateful to  Grigori Olshanski and Alexander Bufetov for helpful discussions.
This research is supported by the grant IDEX UNITI-ANR-11-IDEX-0002-02, financed by Programme ``Investissements d'Avenir'' of the Government of the French Republic managed by the French National Research Agency.


\def\cprime{$'$} \def\cydot{\leavevmode\raise.4ex\hbox{.}} \def\cprime{$'$}

\end{document}